\documentclass[12pt]{article}
\usepackage{amsgen,amsmath,amstext,amsbsy,amsopn,amsfonts,amssymb}
\newtheorem{theorem}{Theorem}
\newtheorem{lemma}[theorem]{Lemma}
\newtheorem{corollary}[theorem]{Corollary}
\newtheorem{proposition}[theorem]{Proposition}
\newtheorem{obs}[theorem]{Observation}

 \newtheorem{defi}[theorem]{Definition}

\newtheorem{exa}[theorem]{Example}

\newtheorem{rem}[theorem]{Remark}

\newtheorem{rems}[theorem]{Remarks}

\newtheorem{ack}[theorem]{Acknowlegment}

\def\bsq{\blacksquare\medskip}

\def\H{\mathcal H}

\def\Hpi{(\pi, {\mathcal H})}

\def\CCC{{\mathbb C}}

\def\RR+{{\mathbb R}^*}

\def\Q_p{{\mathbb Q}_p}

\def\eps{\varepsilon}

\def\La{\Lambda}

\begin{document}

\title{A spectral gap property for subgroups of finite covolume in  Lie groups}
\author{Bachir Bekka and Yves Cornulier}
\date{\today }


\maketitle

\begin{abstract}
Let $G$ be a real  Lie group   and
$H$ a lattice or, more generally, a closed subgroup of finite covolume in $G$.
We show that  the  unitary representation $\lambda_{G/H}$
of $G$  on $L^2(G/H)$ has a spectral gap,
that is, the restriction of $\lambda_{G/H}$ to the
orthogonal of the constants in $L^2(G/H)$
does not have almost invariant vectors.
This answers a question of G.~Margulis.
We give an application  to the spectral geometry
of locally symmetric Riemannian  spaces of infinite volume.

\end{abstract}

\section{Introduction}
Let $G$ a locally compact  group. 
Recall that a unitary representation $\Hpi$
of $G$  {almost has invariant vectors} 
if, for every compact subset $Q$ of $G$ and every 
$\varepsilon>0,$ there exists a unit vector $\xi\in\H$ such that
$\sup_{x\in Q}\Vert \pi(x)\xi-\xi\Vert <\varepsilon.$
If this holds, we also say that the trivial representation $1_G$
is weakly contained in $\pi$ and write $1_G\prec \pi.$

Let $H$ be a closed subgroup of $G$
for which there exists  a $G$-invariant regular Borel measure
$\mu$ on $G/H$  (see \cite[Appendix~B]{BHV} for a criterion
of the existence of such a measure). 
Denote by $\lambda_{G/H}$ the unitary representation
of $G$ given by 
left translation on the Hilbert space $L^2(G/H,\mu)$ 
of the square integrable measurable functions
on the homogeneous space $G/H.$ 
If $\mu$ is finite, we say that $H$ has finite covolume
in $G.$ In this case, the space $\CCC 1_{G/H}$ of the constant functions on $G/H$
is contained in $L^2(G/H,\mu)$ and is 
$G$-invariant as well as its orthogonal
complement 
$$
L^2_0(G/H,\mu)=\left\{\xi\in L^2(G/H,\mu)\, :\, \int_{G/H} \xi(x) d\mu(x)=0\right\}.
$$
In case $\mu$ is infinite, we set $L^2_0(G/H,\mu)= L^2(G/H,\mu).$

Denote by $\lambda_{G/H}^0$ the restriction of $\lambda_{G/H}$
to $L^2_0(G/H,\mu)$ (in case $\mu$ is infinite, $\lambda_{G/H}^0=\lambda_{G/H}$).
We say that $\lambda_{G/H}$ (or $L^2(G/H,\mu)$)  has a \emph{spectral gap}  if  $\lambda_{G/H}^0$   has 
no  almost invariant vectors.
In the terminology of \cite[Chapter III. (1.8)]{Margulis}, $H$ is called weakly cocompact.

By a Lie group we mean a locally compact group $G$ whose connected component of the identity
$G^0$ is open in $G$ and is a real Lie group.
We prove the following result which has been conjectured in \cite[Chapter III.  Remark 1.12]{Margulis}.
\begin{theorem}
\label{Theo1} Let $G$ be a  Lie group   and
$H$ a closed subgroup  with finite covolume in $G$.
Then the  unitary representation $\lambda_{G/H}$ on $L^2(G/H)$ has a spectral gap.
\end{theorem}

It is a  standard fact that $L^2(G/H)$ has a spectral gap
when $H$ is cocompact in $G$ 
(see \cite[Chapter III,  Corollary 1.10]{Margulis}).
When   $G$ is a semisimple Lie group, Theorem~\ref{Theo1}
is an easy consequence  of Lemma~3
in \cite{Bekka}. 
Our proof  is by reduction to these two cases.
The crucial tool for this reduction
is Proposition (1.11) from Chapter III in \cite{Margulis}
(see Proposition~\ref{Margulis} below).
From Theorem~\ref{Theo1} and again from this proposition,  we  obtain the following corollary.
\begin{corollary}
\label{Cor1} Let $G$ be a second countable Lie group, 
$H$ a closed subgroup  with finite covolume in $G$
and $\sigma$ a unitary representation of $H.$ 
Let $\pi={\rm Ind}_H^G \sigma$ be the representation of $G$  induced from $\sigma.$
If $1_H$ is not weakly contained in $\sigma,$ then 
$1_G$ is not weakly contained in $\pi.$
\end{corollary}
Observe that, by continuity of induction,  the converse result is also true: 
if $1_H\prec\sigma,$ 
then $1_G\prec\pi.$ 

From the previous corollary we deduce a spectral gap  result
for some subgroups  of $G$ with infinite covolume.

Recall that a subgroup $H$ of a topological group $G$ is called \emph{co-amenable}
 in $G$ if there is a $G$-invariant mean on the space $C^b(G/H)$
of bounded continuous functions on $G/H$. When $G$ is locally compact,
 this is equivalent to $1_G\prec  \lambda_{G/H}$; this property has been
extensively studied by Eymard \cite{Eymard} for which he refers to as
the amenability of  the homogeneous space $G/H.$
Observe that a normal subgroup $H$ in $G$ is co-amenable in $G$
if and only if the quotient group $G/H$ is amenable.

\begin{corollary}
\label{Cor2} Let $G$ be a second countable Lie group and  
$H$ a closed subgroup  with finite covolume in $G.$
Let $L$ be a closed subgroup of $H$. Assume that $L$ is
not co-amenable in $H.$
Then $\lambda_{G/L}$ has a spectral gap.
\end{corollary}
Corollary~\ref{Cor2} is a direct consequence of Corollary~\ref{Cor1},
since the representation $\lambda_{G/L}$ on
$L^2(H/L)$ is equivalent to the induced representation ${\rm Ind}_H^G \lambda_{H/L}$.

Here is a reformulation of the  previous corollary.
Let $G$ be a  Lie group and  
$H$ a closed subgroup  with finite covolume in $G.$
If a subgroup $L$ of $H$ is co-amenable in $G$, then 
 $L$ is co-amenable in $H.$
Observe that the converse (if $L$ is  co-amenable in $H$, then 
 $L$ is co-amenable in $G$) is true
for  any topological group $G$ and any
closed subgroup $H$ which is co-amenable in $G$
(see \cite[p.16]{Eymard}).

Using  methods  from 
 \cite{Leuz} (see also \cite{Brooks2}),
we obtain  the following consequence for the spectral geometry
of infinite coverings of locally symmetric
Riemannian spaces of finite volume.
Recall that a lattice in the locally compact
group $G$ is a discrete subgroup of $G$ with  finite covolume.

\begin{corollary}
\label{Cor3} Let $G$ be a semisimple Lie group
with finite centre and maximal compact subgroup $K$ and
let $\Gamma$  be a torsion-free lattice
$G.$  Let $\widetilde V$ be a covering of
the  locally symmetric space $V = K\backslash G/\Gamma$.
Assume that the  fundamental group 
of $\widetilde V$ is not co-amenable in $\Gamma.$
\begin{itemize}
 \item[(i)] We have $h(\widetilde V)>0$ for the Cheeger constant
$h(\widetilde V)$ of $\widetilde V.$ 
\item[(ii)]  We have $\lambda_0 (\widetilde V)>0,$ where  $\lambda_0 (\widetilde V)$
is the bottom of the $L^2$ -spectrum of the Laplace--
Beltrami operator on $\widetilde V.$ 
\end{itemize}
\end{corollary}
There is in general no uniform bound for $h(\widetilde V)$
or $\lambda_0 (\widetilde V)$ for all coverings $\widetilde V$. 
However, it was shown in \cite{Leuz} that,
when $G$ has Kazhdan's Property (T), such a
bound exists for \emph{every} locally symmetric space $V = K\backslash G/\Gamma$.
Observe also that if, in the previous corollary,
the fundamental group 
of $\widetilde V$ is co-amenable in $\Gamma$
and has infinite covolume, 
then $h(\widetilde V)=\lambda_0 (\widetilde V)=0,$
as shown in \cite{Brooks1}.


\section{Proofs of Theorem~\ref{Theo1} and Corollary \ref{Cor3}}

The following result of Margulis (Proposition (1.11) in Chapter III of \cite{Margulis})
wil be crucial.
\begin{proposition} \textbf{(\cite{Margulis})}
\label{Margulis}
Let $G$ be a second countable locally compact group,
$H$ a closed subgroup of $G,$ and
$\sigma$ a unitary representation of $H.$ 
Assume that $\lambda_{G/H}$ has a spectral gap and that
$1_H$ is not weakly contained in $\sigma.$ Then 
$1_G$ is not weakly contained in 
${\rm Ind}_H^G \sigma$.
 \end{proposition}

In order to reduce the proof of Theorem~\ref{Theo1}
to the semisimple case,  we will use several times the following
proposition.
\begin{proposition}
\label{Pro1} 
Let $G$ a separable locally compact  group and 
 $H_1$ and $H_2$ be  closed subgroups of $G$
with $H_1\subset H_2$ and such that   
$G/H_2$ and $H_2/H_1$  have  $G$-invariant and $H_2$-invariant
regular Borel measures, respectively.
Assume that the $H_2$-representation
$\lambda_{H_2/H_1}$
 on $L^2(H_2/H_1)$ and that the $G$-representation $\lambda_{G/H_2}$
on $L^2(G/H_2)$ have both spectral gaps. Then the $G$-representation $\lambda_{G/H_1}$
on $L^2(G/H_1)$ has a spectral gap.
\end{proposition}
\begin{proof}
Recall that, for any closed subgroup $H$ of $G,$ the representation
$\lambda_{G/H}$ is equivalent to
the representation ${\rm Ind}_{H}^G 1_H$ induced by the identity representation
$1_H$ of $H$.
Hence, we have, by transitivity of induction,
$$
\lambda_{G/H_1}={\rm Ind}_{H_1}^G 1_{H_1}= {\rm Ind}_{H_2}^G({\rm Ind}_{H_1}^{H_2}1_{H_1})
= {\rm Ind}_{H_2}^G\lambda_{H_2/H_1}.
$$
 We have to consider three cases:

\medskip
\noindent
$\bullet$ {\it First case:} $H_1$ has finite covolume in $G,$
that is, $H_1$ has finite covolume in $H_2$ and $H_2$ has finite covolume in $G.$
Then 
$$\lambda_{G/H_1}^0=\lambda_{G/H_2}^0 \oplus {\rm Ind}_{H_2}^G \lambda_{H_2/H_1}^0.$$
By assumption,  $\lambda_{H_2/H_1}^0$ and $\lambda_{G/H_2}^0$ do not weakly contain
$1_{H_2}$ and $1_G$, respectively.
It follows from Proposition~\ref{Margulis}
that ${\rm Ind}_{H_2}^G \lambda_{H_2/H_1}^0$ does not weakly contain $1_G$.
Hence, $\lambda_{G/H_1}^0$ does not weakly contain $1_G$.

\medskip
\noindent
$\bullet$ {\it Second case:} $H_1$ has finite covolume in $H_2$ and $H_2$ has 
infinite covolume in $G.$ Then
$$\lambda_{G/H_1}=\lambda_{G/H_2}\oplus {\rm Ind}_{H_2}^G \lambda_{H_2/H_1}^0.$$
By assumption, $\lambda_{H_2/H_1}^0$ and $\lambda_{G/H_2}$
do not weakly contain $1_{H_2}$ and $1_{G}$.
As above, using Proposition~\ref{Margulis}, 
we see that  $\lambda_{G/H_1}$ does not weakly contain $1_G$.

\medskip
\noindent
$\bullet$ {\it Third case:} $H_1$ has infinite covolume in $H_2$. 
By assumption,
  $\lambda_{H_2/H_1}$ does not weakly contain
$1_{H_2}.$ By Proposition~\ref{Margulis} again, it follows that
 $\lambda_{G/H_1}={\rm Ind}_{H_2}^G\lambda_{H_2/H_1}$
does not weakly contain
$1_{G}.$ $\bsq$

 \end{proof}

For the reduction of the proof of Theorem~\ref{Theo1}
to the case where $G$ is second countable, we will need
the following lemma.
\begin{lemma}
\label{Pro2} 
Let $G$ be a locally compact group and $H$ a closed subgroup
with finite covolume. The homogeneous space $G/H$ is $\sigma$-compact. 
\end{lemma}
\begin{proof}
 Let $\mu$ be the $G$-invariant regular probability measure
on the Borel subsets of $G/H.$
Choose an increasing  sequence of compact subsets $K_n$ of $G/H$  with
$\lim_n\mu(K_n)=1.$ The set $K=\bigcup_n K_n$ has $\mu$-measure $1$ and is 
therefore dense in $G/H.$ Let $U$ be a compact neighbourhood
of $e$ in $G.$ Then $UK=G/H$ and $UK=\bigcup_n UK_n$ is $\sigma$-compact.$\bsq$
\end{proof}

\goodbreak
\bigskip
\noindent 
\textbf{Proof of Theorem~\ref{Theo1} }

Through several steps the proof will be reduced to the case where $H$ is a lattice 
  in $G$  and where
 $G$ is a connected semisimple Lie group with trivial centre and without compact factors.

\medskip
\noindent
$\bullet$ {\it First step:} we can assume that $G$ is $\sigma$-compact
and hence second-countable. Indeed, let $p:G\to G/H$ be the
canonical projection. Since every compact subset
of $G/H$ is the image under $p$ of some
compact subset of $G$ (see \cite[Lemma~B.1.1]{BHV}),
it follows from Lemma~\ref{Pro2} that there exists
a $\sigma$-compact subset $K$ of $G$ such that 
$p(K)=G/H.$ Let $L$ be the subgroup of $G$ generated
by $K\cup U$ for a neighbourhood $U$ of $e$ in $G.$
Then $L$ is a $\sigma$-compact open subgroup
of $G.$  We show that $L\cap H$ has a finite covolume in $L$ and that 
$\lambda_{G/H}$ has a spectral gap if $\lambda_{L/L\cap H}$
has a spectral gap.

Since $LH$ is open
 in $G,$ the homogeneous space $L/L\cap H$ can be identified 
as  $L$-space with $L H/H$. Therefore $L\cap H$ 
has finite covolume in $L$.
On other hand, the restriction of $\lambda_{G/H}$ to $L$ is 
equivalent to the $L$-representation $\lambda_{L/L\cap H}$,
since $LH/H=p(L)=G/H.$  Hence, if  $\lambda_{L/L\cap H}$
has a spectral gap,
then  $\lambda_{G/H}$ has a spectral gap.

\medskip
\noindent
$\bullet$ {\it Second step:} we can assume that $G$ is connected. Indeed,
let $G^0$ be the connected component of the identity of $G.$ 
We show that $G^0\cap H$ has a finite covolume in $G^0$ and that 
$\lambda_{G/H}$ has a spectral gap if $\lambda_{G^0/G^0\cap H}$
has a spectral gap.

The subgroup $G^0 H$ is
open  in $G$ and has finite covolume in $G$ as it contains
$H.$ It follows that $G^0 H$ has finite index in $G$ since $G/ G^0H$ is discrete.
Hence $\lambda_{G/G^0 H}$ has a spectral gap.

On the other hand, since $G^0H$ is closed in $G$,
the homogeneous space $G^0/G^0\cap H$ can be identified 
as a $G^0$-space with $G^0 H/H$. Therefore $G^0\cap H$ 
has finite covolume in $G^0$. 
The restriction of $\lambda_{G^0 H/H}$ to $G^0$ is 
equivalent to the $G^0$-representation $\lambda_{G^0/G^0\cap H}$.

Suppose now that  $\lambda_{G^0/G^0\cap H}$ has a spectral gap.
Then the $G^0H$-represent\-ation $\lambda_{G^0 H/H}$ has a spectral spectral,
since $L^2_0(G^0H/H)\cong L^2_0(G^0/G^0\cap H)$ as $G_0$-representations.
An application of Proposition~\ref{Pro1}  
with $H_1=H$ and $H_2=G^0 H$ shows  that  $\lambda_{G/ H}$ has a spectral gap.
Hence, we can assume that $G$ is connected.

\medskip
\noindent
$\bullet$ {\it Third step:} we can assume that $H$ is a lattice in $G.$
Indeed, let
$H^0$ be the connected component of the identity of $H$
and let $N_G(H^0)$ be the normalizer of $H^0$ in $G.$
Observe that $N_G(H^0)$ contains $H.$ 
By \cite[Theorem 3.8]{Wang2}, $N_G(H^0)$  is cocompact in 
$G.$  Hence, $\lambda_{G/ N_G(H^0)}$ has a spectral gap.
It follows from the previous proposition that $\lambda_{G/H}$
has a spectral gap if $\lambda_{N_G(H^0)/ H}$ has a spectral gap.

On the other hand, since $H^0$ is a normal subgroup of $H,$ we have
$$L^2_0(N_G(H^0)/H)\cong L^2_0((N_G(H^0)/H^0)/(H/H^0)),$$
as $N_G(H^0)$-representations.
Hence,   $\lambda_{N_G(H^0)/ H}$ has a spectral gap if and only if
 $\lambda_{\overline{N}/\overline{H}}$  has a spectral gap,
where $\overline{N}=N_G(H^0)/H^0$ and $\overline{H}=H/H^0.$

The second step applied to the Lie group $\overline{N}/\overline{H}$ shows that  $\lambda_{\overline{N}/\overline{H}}$
has a spectral gap if   
$\lambda_{\overline{N}^0/(\overline{N}^0\cap \overline{H})}$
has a spectral gap. 
Observe that $\overline{N}^0\cap\overline{H}$ is a lattice in the  connected Lie group
$\overline{N}^0,$ since $\overline{H}$ is discrete and since $H$ has finite covolume
in $N_G(H^0).$

This shows that we can assume that $H$ is a lattice in the connected Lie group $G.$

\medskip
\noindent
$\bullet$ {\it Fourth step:} we can assume that $G$ is a connected semisimple Lie group with no compact factors.
Indeed, let $G=SR$ be a Levi decomposition of $G$, with 
$R$  the solvable radical of $G$ and $S$ a semisimple
subgroup. Let $C$ be the maximal compact normal subgroup
of  $S.$ It is proved in \cite[Theorem~B, p.21]{Wang}
that $HCR$ is closed in $G$ and that $HCR/H$ is compact.
Hence, by the previous proposition,
 $\lambda_{G/H}$ has a spectral gap if  $\lambda_{G/HCR}$
has a spectral gap.

The quotient $\overline{G}=G/CR$ is a connected semisimple Lie group with no compact factors.
Moreover, $\overline {H}=HCR/CR$ is a lattice in $\overline{G}$ 
since $HCR/CR\cong H/ H\cap CR $  is discrete and since $HCR$ has finite covolume
in $G.$ Observe that $\lambda_{G/HCR}$ is equivalent 
to $\lambda_{\overline{G}/\overline {H}}$ as $G$-representation.

\medskip
\noindent
$\bullet$ {\it Fifth step:} we can assume that $G$ has trivial centre.
Indeed, let $Z$ be the centre of $G.$ It is known that $ ZH$ is discrete (and hence closed)
in $G$ (see \cite[Chapter V, Corollary 5.17]{Raghunathan}).
Hence, $ZH/H$ is finite and $\lambda_{ZH/H}$ has a spectal gap.

By the previous proposition,
 $\lambda_{G/H}$ has a spectral gap if  $\lambda_{G/ZH}$
has a spectral gap.
Now, $\overline{G}=G/Z$ is a connected semisimple Lie group with no compact factors
and with trivial centre, 
$\overline {H}=ZH/Z$ is a lattice in $\overline{G}$ 
and $\lambda_{G/ZH}$ is equivalent 
to $\lambda_{\overline{G}/\overline {H}}$.

\medskip
\noindent
$\bullet$ {\it Sixth step:} by the previous steps, we can assume that
$H$ is a lattice in a connected semisimple Lie group $G$ with no compact factors
and with trivial centre. In this case, the claim was proved in Lemma~3 of \cite{Bekka}.
This completes the proof of  Theorem~\ref{Theo1}.$\bsq$

\bigskip
\noindent 
\textbf{Proof of Corollary~\ref{Cor3}}

The proof is identical with the proof
of Theorems~3 and 4 in \cite{Leuz}; we give a  brief
outline of the arguments.  Let $\La$ be the fundamental group
of $\widetilde V$. First, 
it suffices to prove Claims (i) and (ii) for $G/\Gamma$
instead of $K\backslash G/\Gamma$ (see Section~4 in \cite{Leuz}).
 So we assume that $\widetilde V=  G/\La.$ 

Equip $G$ with a right invariant Riemannian metric
and $G/\La$ with the induced Riemannian metric.
Observe that $G/\La$ has infinite volume, since 
$\La$ is of infinite index in $\Gamma$.
Claim (ii) is a consequence of (i), by Cheeger's inequality
$\dfrac{1}{4} h(G/\La)^2\leq \lambda_0(G/\La).$
Recall that the Cheeger constant $h(G/\La)$ of $G/\La$
is the infimum over all numbers $A(\partial \Omega)/V(\Omega)$,
where $\Omega $ is an open submanifold of $G/\La$
with compact closure and smooth boundary $\partial \Omega,$
and where $V(\Omega)$ and $A(\partial \Omega)$ are  
the Lebesgue measures of $\Omega$ and $\partial \Omega$.

To prove Claim~(i), we proceed exactly as in  \cite{Leuz}.
By Corollary~\ref{Cor2}, there exists a compact neighbourhood
$H$ of the identity in $G$ and a constant $\eps>0$ such that  
$$
\leqno{(*)}\qquad\qquad \eps\Vert \xi\Vert\leq \sup_{h\in H}\Vert \lambda_{G/\La}(h)\xi-\xi\Vert,
\qquad\text{for all}\quad \xi \in L^2(G/\La).
$$

Let $\Omega $ be an open submanifold of $G/\La$
with compact closure and smooth boundary $\partial \Omega.$
By \cite[Proposition 1]{Leuz},
we can find an open subset $\widetilde\Omega$ of $G/\La$  
compact closure and smooth boundary such that, for all $h\in H$,
$$
\leqno{(**)}\qquad\qquad V(U_{|h|}( \partial \Omega))\leq C V(\widetilde\Omega) \dfrac{A(\partial \Omega)}{V(\Omega)},
$$ 
where the constant $C>0$ only depends on $H.$
Here, $|h|$ denotes the distance $d_G(e,g)$ of $h$ to the group unit
and, for a subset $S$ of $G/\La$,  $U_{r}(S)$ is the tubular neighbourhood
$$
U_{r}( \partial \Omega)=\{ x\in G/\La\,:\, d_{G/\La} (x, S)\leq r\}
$$
By Inequality $(*)$, applied 
the characteristic function $\chi_{\widetilde \Omega}$  of $\widetilde \Omega$,
there exists $h\in H$ such that
$$
\eps^2 V(\widetilde \Omega)\leq \Vert \lambda_{G/\La}(h)\chi_{\widetilde \Omega}-\chi_{\widetilde \Omega}\Vert^2= V(X),
$$
where
$$
X=\left\{ x\in G/\La\, :\, x\in \widetilde \Omega, hx\notin \partial \widetilde \Omega \right\} \bigcup
\left\{ x\in G/\La\, :\, x\notin \widetilde \Omega, hx\in \partial \widetilde \Omega  \right\}.
$$ 
One checks that $X\subset U_{|h|}( \partial \Omega).$ It follows from
Inequalities $(*)$ and $(**)$ that
$\dfrac{\eps^2}{C}\leq  \dfrac{A( \Omega)}{V(\Omega)}.$
Hence, $0<\dfrac{\eps^2}{C}\leq h(G/\La).$ $\bsq$

\noindent
{\bf Address}

\noindent
IRMAR, UMR 6625 du CNRS, Universit\'e de  Rennes 1, 
Campus Beaulieu, F-35042  Rennes Cedex, France

\noindent
E-mail : bachir.bekka@univ-rennes1.fr, yves.decornulier@univ-rennes1.fr

\end{document}